\documentclass[12pt]{article}

\usepackage{amsmath}
\usepackage{amsthm}
\usepackage{amssymb}
	\usepackage{yfonts}
\usepackage[utf8]{inputenc}
\usepackage{hyperref}
\hypersetup{
	unicode,
	pdfauthor={Author One, Author Two, Author Three},
	pdftitle={A simple article template},
	pdfsubject={A simple article template},
	pdfkeywords={article, template, simple},
	pdfproducer={LaTeX},
	pdfcreator={pdflatex}
}

\usepackage[sort&compress,numbers,square]{natbib}

\theoremstyle{plain}
\newtheorem{theorem}{Theorem}
\newtheorem{corollary}[theorem]{Corollary}
\newtheorem{lemma}[theorem]{Lemma}

\newtheorem{proposition}[theorem]{Proposition}

\theoremstyle{definition}
\newtheorem{definition}[theorem]{Definition}
\newtheorem{example}[theorem]{Example}
\newtheorem{remark}[theorem]{Remark}

\usepackage{graphicx, color}
\usepackage{mathtools}

\graphicspath{{fig/}}

\title{An Introduction to Perfectoid Fields}
\author{Ehsan Shahoseini$^1$\thanks{Mentor, PhD candidate at Tarbiat Modaress University} \and Soroush Pasandideh$^2$\thanks{Mentee, Undergraduate at Shahid Beheshti University} }

\begin{document}
	\maketitle
	
	\begin{abstract}
		This survey is a final project of Twopole DRP  in fall 20201. In this paper we try to understand a tiny part of the vast theory of perfecoid spaces, called perfectoid fields. We start by giving some motivation and historical background. Then we define the notion of a perfectoid field and working through some examples.

	\end{abstract}

	\tableofcontents
	
	\section{Introduction}
		\label{sec:intro}

	There are moments in the history of mathematics that introduce a new idea which helps people to think in somehow different perspective.Those ideas led to a series of new discoveries. One can say that defining the notion of "schemes"  by Alexander Grothendieck was one of those moments and we all know the importence of that defenition in later works, e.g. Delinge's works on Weil's conjectures was inspired heavily by the notion of schemes. One can argue that Peter Scholze's work on defining the notion of perfectoid spaces is also a big moment in some sense. Of course, mathematics does not come from nowhere and like other deep results in mathematics it's heavily rooted in the work of others. We have not much to do with perfectoid spaces in this survey, but by restricting ourself to the notion of perfectoid fields, we try to khow what they are and work through some examples. The theory of perfecoid fields has its roots in perfect fields and affinoid algebras. Scholze describes “phases” of study, applying them to
different topics in number theory. The first phase was giving a correspondence
between geometry in characteristic zero and characteristic $p$, with the goal of
proving Deligne’s weight-monodromy conjecture. The second phase was studing $p$-adic Hodge theory, and how it varies in families. 
Third is the realization of important special cases of “infinite-type rigid geometry”
via perfectoid spaces.  
	
	\subsection{Motivations}
	The original aim of the theory of perfectoid spaces was to prove Deligne’s weight-monodromy
conjecture over $p$-adic fields by reduction to the case of local fields of equal characteristic
$p$, where the result is known. In order to so, the theory of perfectoid spaces establishes a
general framework relating geometric questions over local fields of mixed characteristic with
geometric questions over local fields of equal characteristic. Also, the theory of perfectoid fields merge somehow the notion of perfect fields and affinoid rings. There exists also the notion of deeply ramified fields which was intoduced brfore \cite{GR}, and it's closely related to the theory of perfectoid fields. 
		\newpage

	\subsection{Previous Results}
Perfectoid spaces are a class of algebro-geometric
objects living in the realm of $p$-adic geometry that
were introduced by Peter Scholze \cite{schz12} in his
Ph.D. thesis. Their	 definition is heavily inspired by a
classical result in Galois theory (see Theorem \ref{eq:1}) due
to Fontaine and Wintenberger, and the resulting
theory has already had stunning applications.The miracle of perfectoid spaces is that they provide a functor (called tilting)
between geometric objects in characteristic 0 and in characteristic $p$; in the case of
a single point this essentially recovers a construction of Fontaine and Wintenberger
that underlies $p$-adic Hodge theory. Subsequent work by Scholze has demonstrated
without a doubt that perfectoid spaces are a powerful new tool across many aspects
of algebraic number theory. Perfectoid spaces are  in one sense a tool. It is just a relatively new method to
translate problems from ch. $0$ to ch. $p$, where one may expect to get a solution.
The notion of a perfectoid field is closely related to the notion of a deeply
ramified field. A perfectoid field $K$ is deeply ramified, and
conversely, a complete deeply ramified field with valuation of rank 1 is a perfectoid field.
And also, this theory is a new treatment to some old problems in Algebraic geometry and other related fields; to mention a few one can say that there are tremendous applications of perfectoid spaces in: $p$-adic Hodge theory, rigid analytic geometry, Lubin-Tate spaces, Local Langlands Correspondence and Shimura varieties, see for example:\cite{SH}, \cite{MPS}, \cite{PSAS}, \cite{shlz16}, \cite{BS}, \cite{FS}, \cite{schlzwan}, \cite{schlz13}, \cite{schlz14}, \cite{SMB}, \cite{KL}.

	\begin{theorem}
	\label{eq:1}
	Fontaine-Wintenberger:
	The absolute Galois groups of $\mathbb{Q}_p(p^\frac{1}{p^\infty})$ and $\mathbb{F}_p((t^\frac{1}{p^\infty}))$ are canonically isomorphic. \cite{serre}
	\end{theorem}
	This theorem is one of the foundational cornerstones of $p$-adic Hodge theory. Moreover, it is true in a wide variety of cases. 
	\newpage
	\section{Defintiosns and Theorems}
	The definitions and theorems in ths section are from \cite{schz12}, \cite{SLF}. 
	\\
	\begin{remark}
	Notation: For a ring $R$ of characteristic $p$, we denote the Frobenius endomorphism of $R$ (which sends $x$ to $x^p$) by $\Phi_R$. If there is no ambiguity we use $\Phi$ instead of $\Phi_R$ .
	\end{remark}
	\subsection{Preliminaries}
	Perfect fields are significant because Galois theory over these fields becomes simpler, since separablity is automatically satisfied over these fields.
	\begin{definition}
	\textit{Perfect Field}: A perfect field is a field $F$ such that every algebraic extension is separable. Any field of characteristic zero, such as the rationals or the $p$-adics, and any finite field is a perfect field.
\end{definition}		
\begin{remark}
A field of characteristic $p$ is perfect iff its Frobenius endomorphism is surjective (equivalently, bijective).
\end{remark}
\begin{definition}
A ring of characteristic $p$ is called perfect (resp. semiperfect) if its Frobenius endomorphism is bijective (resp. surjective).
\end{definition}
An example of an imperfect field is the field $\mathbb{F}_q(x)$, since the Frobenius sends $x \mapsto x^p$ and therefore it is not surjective. we can make imperfect fields perfect by a process called \emph{perfection} that makes Forbenious surjective.
	\begin{definition}
	Let $K$ be a field of characteristic $p$. The perfection of $K$, denoted by $K^{perf}$, is the minimal algebraic extention of $K$ which is perfect, and obtained as follow: fix an algebraic closure $\bar{K}$ of K (note that $\bar{K}$ is perfect). Then, $K^{perf} = K \left(\Phi_{\bar{K}}^{-n}(K),\forall n \in \mathbb{N}\right)$. 
	\end{definition}
	\begin{example}
	Let $K=\mathbb{F}_p((t))$. then $K^{perf}=\mathbb{F}_p((t^{\frac{1}{p^\infty}}))$
	\end{example}

		\begin{definition}
	A field which is complete with respect to a nonarchimedean discrete valuation is called a local field
		\end{definition}
		
A local field with characteristic $p>0$ is isomorphic to the field of power series in one variable whose coefficients are in a finite field. A local field of characteristic zero is either $\mathbb{Q}_p$, the field of $p$-adic numbers, or finite extention of it.
	

	\begin{remark}
	 A complete field is a field equipped with a metric and complete with respect to that metric. 
	\end{remark}
	\begin{remark}
	The residue characteristic is the characteristic of the residue field.
	\end{remark}

	\subsection{Perfectoid Fields}
	\begin{definition}
	A \textbf{\emph{perfectoid field}} is a complete nonarchimedean field $K$ of residue
characteristic $p>0$ whose associated rank-1-valuation is nondiscrete, such that the
Frobenius is surjective on $K^\circ/p$.
			\end{definition}
			\begin{lemma}
			let $ |\cdot|: K \rightarrow \Gamma \cup \lbrace 0 \rbrace $ be the unique 		 			rank-1-valuation on $K$, where $\Gamma = |K^\times|$ is chosen minimal. Then 			$\Gamma$ is p-divisible.
			\end{lemma}
			The class of perfectoid fields naturally separates into the fields of 			 			characteristic 0 and
			those of characteristic $p$
\begin{proposition}
In characteristic p, a perfectoid field is the same as a complete
perfect field.
\end{proposition}
		\begin{proof}
	Let $K$ be  perfect $\Rightarrow$ $K \xrightarrow[\text{}]{\Phi} K$ is an automorphism, so $K=K^{\circ}[\frac{1}{\varpi}] \xrightarrow[\text{}]{\Phi} K^{\circ}[\frac{1}{\varpi}]$ is automorphism. Then, $ K^{\circ} \rightarrow  K^{\circ}$ is surjective. As $p=0$, $ K^{\circ}/p =  K^{\circ}$ Thus $K$ is perfectoid.
	\\ 
	Let $K$ be perfectoid $\Rightarrow$ $K^{\circ}/p \xrightarrow[\text{}]{\Phi} K^{\circ}/p$ is surjective, $p=0$ so $K^{\circ} \xrightarrow[\text{}]{\Phi} K^{\circ}$ is surjective. Thus $K=K^{\circ}[\frac{1}{\varpi}] \xrightarrow[\text{}]{\Phi} K=K^{\circ}[\frac{1}{\varpi}]$ is also surjective and hence $K$ is perfect.
	\end{proof}
		Now, we want to define a notion that enable us to ping-pong from
the category of all perfectoid fields to the category of perfectoid fields in characteristic
$p$. This functor is called tilting.
For its first description, choose some element $\varpi \in K^\times$ such that $|p|\leq |\varpi|<1$. Now consider:
	\[ 
			\underset{\Phi}\varprojlim K^\circ / \varpi.
	\]

			 This gives a perfect ring of characteristic $p$. We equip it with the inverse limit topology; note that each $K^\circ / \varpi$ has the discrete topology naturally. 	
		\begin{lemma}
    		\begin{enumerate}
    		\item There is a multiplicative homeomorphism
    			\[ 
				\underset{x\rightarrow x^p}\varprojlim K^\circ  \overset{\cong} \rightarrow \underset{\Phi}\varprojlim K^\circ / \varpi
				\]
    		given by projection. In particular, the right-hand side is independent of
    		$\varpi$ Moreover, we get a map
    			\[ 
				\underset{\Phi}\varprojlim K^\circ / \varpi \rightarrow K^\circ : x \rightarrow x^{\sharp}
				\]
			\item There is an element $\varpi^{\flat} \in 	\underset{\Phi}\varprojlim K^\circ / \varpi$  with $|		 			(\varpi^{\flat})^{\sharp}| = |\varpi|$ .Define
			\[
			K^{\flat}= ( \underset{\Phi}\varprojlim K^\circ / \varpi) \left[(\varpi^{\flat})^{-1}\right].
			\]
			\item There is a multiplicative homeomorphism	
			\[
				K^{\flat}=\underset{x \rightarrow x^p}\varprojlim K		
			\]
			In particular, there is a map $K^{\flat} \rightarrow K$ ,$x \rightarrow x^{\sharp}$ Then $K^{\flat}$ is a perfectoid 				field of characteristic $p$,
			\[ K^{\flat \circ} = \underset{x \rightarrow x^p}\varprojlim K^{\circ} \cong  \underset{\Phi}\varprojlim K^\circ / \varpi,
			 \]
			 and the rank-1-valuation on $K^{\flat}$ can be defined by $|x|_{K^{\flat}} = |x^{\sharp}|_K$ . We have $|K^{\flat \times}| = |K^{\times}|$. Moreover,
			 \[K^{\flat \circ}/ \varpi^{\flat} \cong K^{\circ}/\varpi ,\quad K^{\flat \circ}/\textgoth{m}^{\flat} = K^{\circ}/ \textgoth{m}\]
			 where \textgoth{m}, resp. $\textgoth{m}^{\flat}$, is the maximal ideal of $K^{\circ}$, resp. $K^{\flat \circ}$ 

			\item If $K$ is of characteristic $p$, then $K^{\flat} = K$. We call $K^{\flat}$ the tilt of $K$
		
		\end{enumerate}
		
		\end{lemma}

	\begin{theorem} \label{tilting}
	Let $K$ be a perfectoid field.
	\begin{enumerate}
	\item Let $L$ be a finite extension of $K$. Then $L$ (with its natural topology as a finite-dimensional K-vector space) is a perfectoid field.
	
	\item (Tilting Correspondence): Let $K^{\flat}$ be the tilt of $K$. Then the tilting functor $L \mapsto L^{\flat}$ induces an equivalence
of categories between the category of finite extensions of $K$ and the category of finite
extensions of $K^{\flat}$ . This equivalence preserves degrees.
	
	\end{enumerate}
	\end{theorem}
	\begin{corollary}
	If $K$ is algebraically closed, then $K^{\flat}$ is, too. 
	\end{corollary}
	\begin{corollary} \label{algebraic extensions}
	Let $K$ be a perfectoid field and $L/K$ an algebraic extention. Then $\hat{L}$ is also perfectoid. 
	\end{corollary}
	\begin{proposition}
	Let $K$ be a perfectoid field with tilt $K^{\flat}$. If $K^{\flat}$ is algebraically closed,
	then $K$ is also algebraically closed.
	\end{proposition}
	
	\begin{theorem}
	\label{eq:3}
	Let $K$ be a perfectoid field and $F$ is a finite index subfield of it. Then $F$ is also a perfectoid field. 
	\end{theorem}
	
	\begin{proposition} Let $K$ be an infinitely ramified extention of a local field $F$. If $\hat{K}$ is perfectoid, then $K$ is infinitely wildly ramified over $F$.
	\begin{proof}
	If $K$ is not infinitely wildly ramified, then its value group is not p-divisible. 
	\end{proof}
	\end{proposition}
	\begin{example}
	Let $K$ be a local field and $K^{tame}$ be its maximal tamely ramified extention. Then $\widehat{K^{tame}}$ is not perfectoid.
	\end{example}
\section{Examples}
	\subsection{The characteristic $p$ case: }
	Let $K$ be a chracteristic $p$ perfectoid field with residue field $k$. Then $K$ contains $k\widehat{((t^{\frac{1}{p^\infty}}))}$, where $t$ is any element of $K$ with $0<|t|<1$; i.e. $t$ is a pesudouniformizer of $K$:
	\begin{itemize}
	\item $\mathbb{F}_p \xhookrightarrow[\text{}]{\text{Teichmüller lift}}  K^{\flat}$, since $char K^{\flat} = p$;
	\item $K^{\flat}$ is nonarchimedian field of characteiristic $p$, so $\mathbb{F}_p((t)) \xhookrightarrow [\text{}]{\text{}}  K^{\flat} $;
	\item $K^{\flat}$ is perfect, thus $K^{\flat}$ contains the perfect closure of $\mathbb{F}_p((t))$ which is $\mathbb{F}_p((t^{\frac{1}{p^\infty}}))$;
	\item $K^{\flat}$ is $t$-adically complete so it contain the $t$-adic completion of
	$\mathbb{F}_p((t^{\frac{1}{p^\infty}}))$ which is  $\widehat{\mathbb{F}_p((t^{\frac{1}{p^\infty}}))}$.
	\end{itemize}
	\subsection{The $\widehat{\mathbb{Q}_p(p^{\frac{1}{p^\infty}}})$ case }
	\begin{example}
	Let $K = \widehat{\mathbb{Q}_p(1/p^\infty)}$. We want to show that $K$ is perfectoid, and then determine its tilt. \\
	A field shoud satisfy three condtion to be perfectoid field
	\begin{itemize}
	
	\item $K$ should be complete,
	\item $K$ be nondiscretly valued,
	\item $K^{\circ}/p$ should be semiperfect.
	\end{itemize}
	The first condition is true by defenition. Also, we have $|K^{\times}| = \bigcup_{n=1}^{\infty} 1/p^n \mathbb{Z} = \mathbb{Z}[1/p]=\mathbb{Z}_{\lbrace p \rbrace}$, so it is nondiscrete and thus secend condition holds too. Now we should check the semiperfectness condition. Note that completeion does not change the residue ring. We have:
	\[ K^{\circ}/p = \frac{\widehat{\mathbb{Z}_p [p^{1/p^\infty}]}}{p} \cong 	\frac{\mathbb{Z}_p [p^{1/p^\infty}]}{p} \cong  \frac{\frac{\mathbb{Z}_p [t^{1/p^\infty}]}{<t-p>}}{p}  \cong \frac{\mathbb{F}_p [t^{1/p^\infty}]}{<t-p>} = \frac{\mathbb{F}_p [t^{1/p^\infty}]}{<t>} \]
	so the third condition is also true because trivially  $\frac{\mathbb{F}_p [t^{1/p^\infty}]}{<t>}$ is semiperfect. Thus $K$ is a perfectoid field. \\
	Now we want to find out what is $K^{\flat}$ ( the tilt of $K$):	We see that $K^{\flat} = \underset{x \mapsto x^p}\varprojlim K $
	contains the element $t:=(p,p^{1/p},p^{1/p^2},\dots)$ and $|t|_{K^{\flat}} = |p|_K$ so $t$ is a pseudouniformizer of $K^{\flat}$. Since $K^{\flat}$ is a perfectoid field of char.	 $p$ and residue field $\mathbb{F}_p $     
	$ (k_K = k_{K^\flat})$, so $K^{\flat}$ contains $\widehat{\mathbb{F}_p((t^{1/p^\infty}))}$. We want to show that $K^{\flat} =\widehat{\mathbb{F}_p((t^{1/p^\infty}))}$. \\
	
	We know that: 
	\[ K^{\flat\circ} = \underset{\Phi}\varprojlim K^{\flat}/p ,\quad K^{\flat} =K^{\flat\circ}[\frac{1}{t}]. 
	\]
	We have:
	\[K^{\circ}/p \cong \frac{\mathbb{F}_p[t^{\frac{1}{p^\infty}}]}{t} \Rightarrow K^{\flat\circ} = \underset{\Phi}\varprojlim K^{\circ}/p  \cong {\underset{\Phi}\varprojlim \frac{\mathbb{F}_p[t^{\frac{1}{p^\infty}}]}{t}} \cong  \widehat{\mathbb{F}_p[[t^{\frac{1}{p^\infty}}]]}.
	\]
	so $ K^{\flat\circ} = \widehat{\mathbb{F}_p[[t^{\frac{1}{p^\infty}}]]} \Longrightarrow K^{\flat} = \widehat{\mathbb{F}_p((t^{\frac{1}{p^\infty}}))} $ and the example is complete now!
	
	\end{example}
	
	\begin{remark}
	If $R$ is a perfect ring of char. $p$ and $f \in R$ is a nonzerodivisor, then $\underset{\Phi}\varprojlim \frac{R}{f^n}=f$-adic completion of $R$. For example ${\underset{\Phi}\varprojlim \frac{\mathbb{F}_p[t^{\frac{1}{p^\infty}}]}{t}} =  \widehat{\mathbb{F}_p[[t^{\frac{1}{p^\infty}}]]}$.
	\end{remark}


	\subsection{The $\widehat{\mathbb{Q}_p(\mu_{p^\infty}})$ case }
	\begin{example}
	Let $K = \widehat{\mathbb{Q}_p(\mu_p^\infty)}$.  We will show that $K$ is perfectoid and then find its tilt.
	 \\
	
	This extention is khown as the cyclotomic perfectoid field and sometimes denoted by $\mathbb{Q}_{p}^{cycl}$. Like the previous example, we should check three conditions: completeness, semiperfectness and being nondiscretely valued.
	\begin{itemize}
	\item by defenition $K$ is complete, 
	\item we see that the value group of $K$ is:
	\[ |K^\times| = \bigcup_{n=1}^{\infty}\frac{1}{p^n}\mathbb{Z}
	\]
	 which is nondiscrete.
	\item now we should check the semiperfectness condition:
	
	\[ \frac{K^{\circ}}{p} =\frac{\widehat{\mathbb{Z}_{p}[\mu_{p}^{\infty}]}}{p} = \frac{\mathbb{Z}_{p}[\mu_{p}^{\infty}]}{p}
	=\]
	
	\[=
	\frac{\mathbb{Z}_p[\zeta_{p},\zeta_{p^2},\zeta_{p^3}, \dots]}{p} \cong \frac{\frac{\mathbb{Z}_p[x_1,x_2,x_3,\dots]}{<x_1^{p-1}+x_1^{p-2}+\dots +x_1+1,x_2^{p}-x_1,x_3^{p}-x_2, \dots>}}{p} \]
	\[ \cong \frac{\mathbb{F}_p[x,x^{1/p},x^{1/p^2},\dots]}{<x^{p-1}+x^{p-2}+ \dots +x+1 >} =\frac{\mathbb{F}_p[x^\frac{1}{p^\infty}]}{<x^{p-1}+x^{p-2}+ \dots +x+1 >}=\]
	
	\[\frac{\mathbb{F}_p[x^\frac{1}{p^\infty}]}{\frac{x^{p}-1}{x-1}}=\frac{\mathbb{F}_p[x^\frac{1}{p^\infty}]}{\frac{(x-1)^{p}}{x-1}}=\frac{\mathbb{F}_p[x^\frac{1}{p^\infty}]}{(x-1)^{p-1}} = \frac{\mathbb{F}_p[(1+t)^{1/p^\infty}]}{t^{p-1}}
	\]

	\[=\frac{\mathbb{F}_p[(1+t^{1/p^\infty})]}{t^{p-1}} = \frac{\mathbb{F}_p[t^{1/p^\infty}]}{t^{p-1}}.
	\]
	The last expression is semiperfect trivially, and we have shown that $K$ is perfectoid by semiperfectness of $K^{\circ}/p \cong \frac{\mathbb{F}_p[t^{1/p^\infty}]}{t^{p-1}} $. 		
	\end{itemize}
	Now we want to find out the tilt $K^{\flat}$ of $K$. We have:
	\[ K^{\flat\circ}=\underset{\Phi}\varprojlim K^{\flat}/p \cong \underset{\Phi}\varprojlim \frac{\mathbb{F}_p[t^{1/p^\infty}]}{t^{p-1}}= \widehat{\mathbb{F}_p[[t^{1/p^\infty}]]} \Rightarrow K^{\flat} = \widehat{\mathbb{F}_p((t^{1/p^\infty}))}
	\]
	
	Considering above computations, the example is done!
	\end{example}

\subsection{The $\mathbb{C}_p$ case }
\begin{example}
Let $\mathbb{C}_p:=\widehat{\bar{\mathbb{Q}}}_p$. We want to show that it is perfectoid and determine its tilt. We know that $\mathbb{C}_p$ is also the completion of algebraic closure of $\mathbb{Q}_p^{cycl}$ (we can use $\widehat{\mathbb{Q}_p(p^{\frac{1}{p^\infty}}})$ instead of $\mathbb{Q}_p^{cycl}$, too), so by \eqref{algebraic extensions} it is a perfectid field. Also, by titling correspondence (part $2$ of \eqref{tilting}), we have that $\mathbb{C}_p^\flat$ is equal to the completion of algebraic closure of $\mathbb{Q}_p^{cycl,\flat} = \widehat{\mathbb{F}_p((t^{1/p^\infty}))}$, which itself is equal to $\widehat{\overline{\mathbb{F}_p((t))}}$, thus we have $\mathbb{C}_p^\flat = \widehat{\overline{\mathbb{F}_p((t))}}$.

\end{example}

\subsection{The Converse of Tilting Correspondence (Tilting Correspondence for Subfields) Doesn't Hold}
	In this section we state Problem 38 of The Nonarchimedean Scottish Book \cite{NAS}. Before going to it, we state two theorems that we need them.
	\begin{theorem}
		\label{eq:2}

	(Ax-Sen-Tate \cite{AX}): Let $K$ be a field which is complete with respect to a valuation of rank one and $C$ be the completion of an algebraic closure of $K$. By continuity, the absolute Galios group $G_k \coloneqq Gal(K^{sep}/K)$ of K acts on $C$. Then, the fixed field in $C$ of this action equals to the completed perfection of $K$; i.e. $(\hat{\bar{K}})^G = C^G = \widehat{K^{perf}}$.
		\end{theorem}

	\begin{remark}
	In fact $\ref{eq:2}$ says that the two operations \textit{completion} and \textit{taking fixed field} commute with each other; i.e $(\hat{\bar{K}})^G = \widehat{(\bar{K})^G}$. Also, if $L/K$ is an algebraic Galios extention with $Gal(L/K) \cong  G_1$, then we have $(\widehat{L^{perf}})^{G_1} = \widehat{(L^{perf})^{G_1}} = \widehat{K^{perf}}$; i.e. \ref{eq:2} holds for any algebraic extention $L$ of $K$, after perfection. The reason is as follow: We know that $\bar{L}=\bar{K}$. Let $G=G_K$ and $H=G_L$. Then we have $G_1 \cong \frac{G}{H}$. If we apply \ref{eq:2} to  $K$ and $L$, we have 
$(\hat{\bar{K}})^G =\widehat{K^{perf}}$ and $(\hat{\bar{L}})^G =\widehat{L^{perf}}$. So:
\[
(\widehat{L^{perf}})^{G_1} = ((\hat{\bar{L}})^H)^{G/H} = (\hat{\bar{L}})^G = (\hat{\bar{K}})^G = \widehat{K^{perf}}.
\]

	\end{remark}
	
	\begin{theorem}
	
	(Kedlaya-Temkin \cite{KT}): There exist nonsurjective endomorphisms of $\mathbb{C}_p^{\flat}$.
	\end{theorem}

	\begin{example}
	
	Let $K$ be a perfectoid field of characteristic $0$. Let $E$ be a perfectoid subfield of the tilt $K^{\flat}$ of $K$. Is $E$ necessarily the tilt of a perfectoid subfield of $K$? \\
(due to Ofer Gabber, see \cite{NAS})	To obtain a negative answer with $[K^{\flat}:E]$ finite, let $K$ be the completion of $\mathbb{ Q}_p(\mu_{p^\infty})$. The tilt of K is the completed perfect closure of $\mathbb{F}_p((t))$. For any integer $m>1$ coprime to p, the completed perfect closure of $\mathbb{F}_p((t^m))$ is a perfectoid subfield E of K of index m (by \ref{eq:3}). If E is the tilt of a subfield F of K, then by Ax-Sen-Tate, $F$ is itself the completion of a subfield of $\mathbb{Q}_p(\mu_{p^\infty})$; we can rule this out by requiring m to be coprime to $p(p-1)$.
	\\
	(due to Kiran S. Kedlaya and Ehsan Shahoseini, see \cite{NAS}) One can also obtain negative answers using the existence of nonsurjective endomorphisms of $\mathbb{C}_p^{\flat}$. If $K = \mathbb{C}_p$, then we can find a subfield $E$ of $K^{\flat}$ which is again the completed algebraic closure of a power series field, but which is strictly smaller than $K^{\flat}$. Such a field cannot be the tilt of a perfectoid subfield of $K$: such a field would itself have to be algebraically closed (and complete), but there are no complete algebraically closed fields between $\mathbb{Q}_p$ and $\mathbb{C}_p$.
	 
	\end{example}
	
		\newpage



\begin{thebibliography}{9}
	\bibliographystyle{plain}
	 \bibliography{Mybibtex}
	 
\bibitem{serre}
Jean-Marc Fontaine and Jean-Pierre Wintenberger, Extensions algébrique et corps des normes des
extensions APF des corps locaux, C. R. Acad. Sci. Paris
Sér. A–B 288(8) (1979), A441–A444.
\bibitem{schz12}
Peter Scholze. Perfectoid Spaces. Publ. math. de l’IHÉS 116 (2012),
no. 1, 245-313.
\bibitem{SH}
Peter Scholze.\href{http://mathoverflow.net/questions/65729/what-are-perfectoid-spaces}{http://mathoverflow.net/questions/65729/what-are-perfectoid-spaces}
\bibitem{MPS}
Bruno Chiarellotto, Motivation for the theory of Perfectoid Spaces

\bibitem{PSAS}
Peter Scholze, Perfectoid spaces: a survey, Current developments in mathematics 2012,Int. Press, Somerville, MA, 2013, pp. 193–227
\bibitem{NAS}
\href{https://scripts.mit.edu/~kedlaya/wiki/index.php?title=The_Nonarchimedean_Scottish_Book#Problem_38}{The Nonarchimedean Scottish Book, Problem 38} 
\bibitem{SLF}
Jean-Pierre Serre, Local fields, Graduate Texts in Mathematics, 67 (2 ed.), Springer-Verlag, 1979 
\bibitem{AX}
James Ax, Zeros of polynomials over local fields—The Galois action,Journal of Algebra Volume 15, Issue 3, July 1970, Pages 417-428
\bibitem{schlz13}
Peter Scholze,On the $p$-adic cohomology of the Lubin-Tate tower, Annales scientifiques de l'ENS, 2018, 811-863

\bibitem{schlzwan}
Peter Scholze, Jared Weinstein, Moduli of p-divisible groups, Cambridge Journal of Mathematics Volume 1, Number 2, 145–237, 2013

\bibitem{schlz14}
Peter Scholze, On torsion in the cohomology of locally symmetric varieties, Annals of Mathematics, Pages 945-1066  Volume 182 (2015), Issue 3


\bibitem{SMB}
Bhargav Bhatt, Matthew Morrow, Peter Scholze, Integral $p$-adic Hodge theory,  Publ.math.IHES 128, 219–397 (2018).
\bibitem{GR}
Ofer Gabber, Lorenzo Ramero, Almost Ring Theory, Springer  Lecture Notes in Mathematics book series, volume 1800, 2003
\bibitem{KL}
Kiran S. Kedlaya, Ruochuan Liu, Relative $p$-adic Hodge theory : Foundations, Astérisque, 2015
\bibitem{KT}
Kiran S. Kedlaya, Michael Temkin, Endomorphisms of power series fields and residue fields of Fargues-Fontaine curves,  Proc. Amer. Math. Soc. 146 (2018), 489-495 
\bibitem{FS}
Laurent Fargues, Peter Scholze, Geometrization of the local Langlands correspondence, \href{https://arxiv.org/abs/2102.13459}{arXiv:2102.13459} 
\bibitem{BS}
Bhargav Bhatt, Peter Scholze, Prisms and Prismatic Cohomology, \href{https://arxiv.org/abs/1905.08229}{arXiv:1905.08229v3}
\bibitem{shlz16}
Peter Scholze, $p$-adic Hodge theory for rigid-analytic varieties, Forum of Mathematics Pi, Cambridge University Press, 2013

\end{thebibliography}
\end{document}